\documentclass[11pt]{ article}
\usepackage{amsthm,amstext,color}
\usepackage{amsmath}
\usepackage{graphicx}
\usepackage{amsfonts}
\usepackage{amssymb}
\usepackage{array}
\theoremstyle{plain}
\newtheorem{theorem}{Theorem}[section]

\theoremstyle{definition}
\newtheorem{definition}[theorem]{Definition}
\newtheorem{definitions}[theorem]{Definitions}
\newtheorem{note}[theorem]{Note}

\newtheorem{example}[theorem]{Example}
\newtheorem{remark}[theorem]{Remark}

\newtheorem{concluding remark}[theorem]{Concluding remark}

\newtheorem{notations}[theorem]{Notations}
\numberwithin{equation}{theorem}

\date{}
\title{\bf On Union of Regular Near-rings}\vspace{.25 in}
\author{{\bf Rajlaxmi Mukherjee$^1$,} {\bf Tuhin Manna$^2$,} {\bf Kamalika Chakraborty$^3$,} {\bf Sujit Kumar Sardar$^4$}\\
{\tt$^1$ Garhbeta College, Paschim Medinipur-721127, India}\\
{\tt$^{2}$ Krishnagar Government College, Krishnagar, Nadia-741101, India}\\
{\tt$^{3}$ Techno India University, Kolkata-700091, India}\\
{\tt$^{4}$ Jadavpur University, Jadavpur, Kolkata-700032, India}\\
{\tt $^1$ju.rajlaxmi@gmail.com, }
{\tt$^2$iamtuhinmanna@gmail.com,}\\
{\tt$^3$kchakrabortyjumath@gmail.com, }
{\tt$^4$ sksardarjumath@gmail.com}}
\begin{document}
\maketitle
\begin{abstract}
`A semigroup is completely regular if and only if it is a union of groups'- an analogue of this structure theorem of completely regular
semigroup has been obtained in the setting of seminearrings in [\cite{Tuhin2}, Mukherjee (Pal) et al., Semigroup Forum (2018)].
In it, a class of seminearrings (called generalized left completely regular seminearrings, abbreviated as GLCR) has been characterized as a union of near-rings.
This work has been extended in the present article to characterize the seminearrings which are union of various types (regular, completely regular, inverse, Clifford) of regular near-rings.
\end{abstract}
\section{Introduction}
A seminearring $(S,+,\cdot)$ consists of two semigroups $(S,+)$ and $(S,\cdot)$ where the multiplication is distributive over addition from one side. Throughout our work it is considered to be distributive from the right.
It is clear from the definition that it generalizes the notion of semiring as well as of near-ring.
It differs with semiring by one distributive property and it differs with near-ring by the existence of additive inverse.
Seminearring is important not only from its own interesting feature as common generalization of semiring and near-ring (see \cite{Golan book,I-congruence,KK1,Pilz,Weinert}) but also from its various applications (see \cite{automata,Baeten,BW,Boyket1,Boykett,desh,dro,KK4}). It is well known that the process algebra is an active area of research in computer science.
From last century, many process algebras have been formulated, extended with data, time, mobility, probability and stochastic (see \cite{Baeten, BW}).
A process algebra is based upon seminearrings where `+' is idempotent and commutative.
Seminearring is also a useful tool in the study of reversible computation \cite{Boykett}.
It also appears in generalized linear sequential machines.
In \cite{KK4}, the authors obtained a necessary condition to test the minimality of the machines using $\alpha$-radicals.
Desharnais and Struth \cite{desh}, Droste et al. \cite{dro}, Armstrong et al. \cite{automata}, Rivas et al. \cite{rivas}, Jenila et al. \cite{JD} utilized the concept of seminearring in various applications.

One aspect of developing the seminearring theory is to manipulate with additive semigroup structure so as to investigate its difference from a near-ring.
This aspect is more or less motivated by the similar study of semirings to find its gap from rings (see \cite{MKS}).
The said aspect has been dealt with the help of nice structure theorems of semigroup theory resulting into \cite{RMCR,Tuhin1,Tuhin2,SKS}.
In \cite{Tuhin2}, among other things, the analogue of ``a semigroup is completely regular if and only if it is a union of groups" was obtained.
In \cite{RMCR}, the analogue of ``a semigroup is completely regular if and only if it is a semilattice of completely simple semigroups" was obtained.
While passing from semigroup theory to seminearring theory, we follow the dictionary given below.
\\
$\begin{array}{cc}
  \underline{\textbf{Semigroup}} & ~~~~~~~~~~~~~~~~~~~~~~~~~\underline{\textbf{Seminearring}}\\
  semilattice &  ~~~~~~~~~~~~~~~~~~~~~~~~~~bisemilattice\\
  completely~~ simple &  ~~~~~~~~~~~~~~~~~~~~~~~~~~ (left)~~ right~~ completely~~ simple \\
  union~~ of~~ groups & ~~~~~~~~~~~~~~~~~~~~~~~~~~ union~~ of~~ near-rings \\
  completely~~ regular & ~~~~~~~~~~~~~~~~~~~~~~~~~~~(generalized)~~ left~~(right)~~completely~~ regular.
\end{array}$
\\
\\
 In semigroups, both union of groups and semilattice of completely simple semigroups represent the same class of semigroups which is completely regular semigroups. So according to the above dictionary, both union of near-rings and bisemilattice of left completely simple seminearrings are expected to characterize the left completely regular seminearring (abbreviated as LCR, see Definitions \ref{def1}). But Examples 2.8 and 2.17 \cite{Tuhin2} show that the above two classes do not land in the same place. So the two classes were given different names. The bisemilattice of left completely simple seminearrings are called (Theorem 2.23 \cite{RMCR}) the left completely regular (LCR) seminearrings and the union of near-rings are called (Theorem 2.19 \cite{Tuhin2}) the generalized left completely regular seminearrings (abbreviated as GLCR, see Definitions \ref{def1}). The lack of one distributive property in a near-ring prevents zero to be absorbing from both sides. The near-rings, in which zero is absorbing, are called zero-symmetric. The union of zero-symmetric near-rings represents (Theorem 2.25 \cite{Tuhin2}) the class of seminearings which are simultaneously GLCR and GRCR (abbreviation for generalized right completely regular seminearring, see Definitions \ref{def1}). Note here that what is GLCR for right distributive seminearrings that is GRCR for left distributive seminearrings.\\
 The present article is a continuation of the study of GLCR seminearrings ({\it i.e.}, union of near-rings) accomplished in \cite{Tuhin2}. The study is initiated by an attempt to investigate the situation that occurs on replacing the component near-rings of GLCR seminearrings by regular near-rings. In this direction, the first easy observation is that the seminearrings represented by union of regular near-rings are multiplicatively regular (Note that by definition a near-ring is additively regular but by definition a seminearring is neither additively regular nor multiplicatively regular) {\it i.e.}, a GLCR seminearring, whose component near-rings are regular, is multiplicatively regular.
 But the converse is not true. In Example \ref{multiplicatively regular GLCR seminearring}, we provide a GLCR seminearring which is multiplicatively regular but not all component near-rings are regular. In Example \ref{example multiplicatively regular GLCR semiring}, we illustrate that the situation does not improve even after adding the condition of being zero-symmetric. This prompts us to obtain Theorem \ref{theorem mult reg} where we find necessary and sufficient conditions, on GLCR multiplicatively regular seminearings, for component near-rings to be regular. Surprisingly the situation becomes nice when we put restriction on regularity of component near-rings.
 In Theorem \ref{theorem mult inverse} we establish that a GLCR siminearring is multiplicatively inverse if and only if each component near-ring is inverse ({\it i.e.}, multiplicative reduct of the near-ring is an inverse semigroup). The situation is exactly similar for Clifford ({\it cf.} Theorem \ref{theorem Clifford}) and completely regular case ({\it cf.} Theorem \ref{theorem 1}).

\section{Main Results}

\begin{definition}\label{definition of additively regular, inverse seminearring, multiplicatively regular and inverse seminearring}
Suppose $(S,+,\cdot)$ is a seminearring.
Then $S$ is said to be an {\it additively regular seminearring} ({\it additively inverse seminearring}) if $(S,+)$ is a regular semigroup (resp. an inverse semigroup)
({\it cf.} \cite{I-congruence, RMCR}).
Similar are the definitions of {\it additively completely regular}, {\it additively completely simple} and {\it additively Clifford seminearrings} ({\it cf.} \cite{RMCR}).
For semigroup theoretic counterparts of these notations we refer to \cite{Reilley completely regular}.

$S$ is said to be a {\it multiplicatively regular seminearring} ({\it multiplicatively inverse seminearring}) if $(S,\cdot)$ is a regular semigroup (resp. an inverse semigroup)
({\it cf.} \cite{I-congruence}).

We define $S$ to be a {\it multiplicatively completely regular seminearring} if the multiplicative reduct of $S$ is completely regular. We define the notion of {\it multiplicatively Clifford seminearrings} analogously.
\end{definition}

\begin{note}\label{note}
A near-ring is always not only an additively regular, but also an additively inverse as well as an additively completely regular and an additively Clifford seminearring.
So in what follows, we call a near-ring, whose multiplicative reduct is an inverse or a completely regular or a Clifford semigroup, to be an {\it inverse near-ring} or a {\it completely regular near-ring} or a {\it Clifford near-ring}, respectively, instead of calling them a {\it multiplicatively inverse near-ring} or a {\it multiplicatively completely regular near-ring} or a {\it multiplicatively Clifford near-ring}.
\end{note}

\begin{notations}\label{notation1}
Throughout this paper, unless mentioned otherwise,
\begin{itemize}
\item[(i)] for a seminearring $S$, $E^+(S)$ denotes the set of all additive idempotents;
\item[(ii)] in an additively completely regular seminearring $S$, for $a\in S$, an element $x\in S$ satisfying $a+x+a=a$ and $a+x=x+a$
is denoted by $x_{a}$;
\item[(iii)] $\mathcal{L}^+$, $\mathcal{R}^+$, $\mathcal{H}^+$  and $\mathcal{J}^+$ denote the Green's relations $\mathcal{L}$, $\mathcal{R}$, $\mathcal{H}$ and $\mathcal{J}$ on the semigroup $(S, +)$, the additive reduct of the seminearring $S$;
\item[(iv)] in a seminearring $S$, $\mathcal{L}^+_a$, $\mathcal{R}^+_a$, $\mathcal{H}^+_a$  and $\mathcal{J}^+_a$ denote the $\mathcal{L}^+$, $\mathcal{R}^+$, $\mathcal{H}^+$  and $\mathcal{J}^+$ classes of $a\in S$, respectively;
\item[(v)] in an additively completely regular seminearring $S$, for each $a\in S$, $(\mathcal{H}_{a}^+,+)$ is a group. The identity element of this group is denoted by $0_{\mathcal{H}_{a}^+}$.
\end{itemize}
\end{notations}

\begin{definitions}\label{def1}\cite{Tuhin2}
An additively completely regular seminearring $(S, +, \cdot)$ is called a {\it left completely regular (LCR)} ({\it right completely regular (RCR)}) if 
\begin{itemize}
\item[(i)] for each $a \in S$ there exists an $x_{a}\in S$ ({\it cf.} Notation \ref{notation1}$(ii)$) satisfying
$(a + x_{a})a$ $=$ $a+ x_{a}$ (resp., $a(a + x_{a})$ $=$ $a+ x_{a}$), and
\item[(ii)] $be$ $\mathcal{J}^{+}$ $eb$ for all $b\in S$ and for all $e\in E^{+}(S)$.
\end{itemize}
If we remove condition $(ii)$ from the definition of a left completely regular (LCR) (resp. right completely regular (RCR)) seminearring then what we get is called a generalized left completely regular (GLCR) (resp. generalized right completely regular (GRCR)) seminearring.
\end{definitions}

In Theorem $2.19$ \cite{Tuhin2}, generalized left completely regular (GLCR) seminearrings are characterized as union of near-rings. In Theorem $2.25$ \cite{Tuhin2}, the seminearrings which are GLCR as well as GRCR are characterized as union of zero-symmetric near-rings. It is clear that if the component near-rings are regular then the multiplicative reduct of the seminearring is regular. We provide below two examples (the second one is the zero-symmetric case) illustrating that the converse is not true {\it i.e.}, the seminearring is multiplicatively regular but not all component near-rings are regular.

\begin{example}\label{multiplicatively regular GLCR seminearring}
Consider `$+$' on $T=\{u,a,b,c\}$ defined as follows
\begin{center} 
\begin{tabular}{|c|c c c c|}
  \hline
  $+$ & $u$ & $a$ & $b$ & $c$\\
  \hline
  $u$ & $u$ & $a$ & $b$ & $c$\\

  $a$ & $a$ & $a$ & $a$ & $a$\\

  $b$ & $b$ & $b$ & $b$ & $b$\\

  $c$ & $c$ & $b$ & $a$ & $u$\\
  \hline
\end{tabular}
\end{center}
$(T,+)$ is the semigroup considered in Example $2.8$ of \cite{Tuhin2}.
Let us define `$\ast$' on $T$ by $x\ast y=x$ for all $x$, $y\in T$.
It is a matter of routine verification to check that $(T,+,\ast)$ is a GLCR seminearring which is also multiplicatively regular.
Consider the regular ring $(M_{2}(\mathbb{R}), +, \cdot)$ and one of its right ideal $I$ $=$ $\Big\{\begin{pmatrix}
             a & b \\
             0 & 0
        \end{pmatrix}$ $|a, b \in \mathbb{R}\Big\}$.
Now let us define $S$ $=$ $S_{1}\cup S_{2} \cup S_{3}$ $\subset$ $T \times M_{2}(\mathbb{R})$ (seminearring direct product of $T$ and $M_{2}(\mathbb{R})$), where $S_{1}$ $=$ $\{u, c\}\times I$, $S_{2}$ $=$ $\{a\}\times M_{2}(\mathbb{R})$ and $S_{3}$ $=$ $\{b\}\times M_{2}(\mathbb{R})$. Clearly $S_{1}$, $S_{2}$ and $S_{3}$ are near-rings. In view of Theorem $2.19$ of \cite{Tuhin2}, $S$ is a GLCR seminearring.
Let $s \in S$. Suppose $s \in S_{2}$ or $s \in S_{3}$. Then $s =(x,A)$ for some $A\in M_{2}(\mathbb{R})$ and $x \in \{a, b\}$. Since $M_{2}(\mathbb{R})$ is regular, there exists $B \in M_{2}(\mathbb{R})$ such that $ABA=A$. Then $t= (x, B) \in S$ satisfying $sts=s$. Now if $s\in S_{1}$ then $s = (y,X)$ where $X \in I$ and $y \in \{u, c\}$.
Let $X= \begin{pmatrix}
          a & b \\
          0 & 0
       \end{pmatrix}$ where $a, b \in \mathbb{R}$.
Now if $a\neq 0$, take $t= (y,Y)$ where
$Y = \begin{pmatrix}
     \frac{1}{a} & 0 \\
               0 & 0
       \end{pmatrix}$.
Then $t \in S$ satisfying $sts=s$. Similarly if $b\neq 0$, take $t= (y_{1},Z)$ where
$Z = \begin{pmatrix}
     0 & 0 \\
     \frac{1}{b} & 0
\end{pmatrix}$ and $y_{1} \in \{a, b\}$. Then $t \in S$ satisfying $sts=s$. Hence $S$ is multiplicatively regular.
Let $s_{1}= (y, C)\in S_{1}$ where $C=\begin{pmatrix}
 0 & 1 \\
  0 & 0
\end{pmatrix}$ and $y \in \{u, c\}$.
 Now for any $D=\begin{pmatrix}
                   a & b \\
                   0 & 0
                 \end{pmatrix}\in I$,
$CDC$ $=$ $\begin{pmatrix}
         0 & 0 \\
         0 & 0
    \end{pmatrix}$.
 Thus there is no $D\in I$ which satisfy $CDC=C$. Hence there exists no such $t \in S_{1}$ such that $s_{1}ts_{1}=s_{1}$. So, $S_{1}$ is not a multiplicatively regular near-ring.
\end{example}
\begin{example}\label{example multiplicatively regular GLCR semiring}
Let us define `$+$' on $L=\{\alpha, \beta\}$ as follows.
\begin{center}
\begin{tabular}{c|c c }

  $+$ & $\alpha$ & $\beta$\\
  \hline
  $\alpha$ & $\alpha$ & $\beta$\\

  $\beta$ & $\beta$ & $\beta$\\

\end{tabular}
\end{center}
$(L,+)$ is nothing but a semilattice with two elements.
Let us define `$\ast$' on $L$ by $x\ast y=x$ for all $x$, $y\in L$. It is a matter of routine verification that $(L,+,\ast)$ is a semiring which is also multiplicatively regular.
Consider the regular ring $(M_{2}(\mathbb{R}), +, \cdot)$ and one of its right ideal $I$ $=$ $\Big\{\begin{pmatrix}
     a & b \\
     0 & 0
    \end{pmatrix}|a, b \in \mathbb{R}\Big\}$.
Now let us define $S$ $=$ $T_{1}\cup T_{2}$ $\subset$ $L \times M_{2}(\mathbb{R})$ (semiring direct product of $L$ and $M_{2}(\mathbb{R})$), where $T_{1}$ $=$ $\{\alpha\}\times I$ and $T_{2}$ $=$ $\{\beta\}\times M_{2}(\mathbb{R})$.
Clearly $T_{1}$ and $T_{2}$ are skew-rings.
In view of Theorem $1.4$ of \cite{MKS}, $S$ is a completely regular semiring.

Let $s \in S$.
Suppose $s \in T_{2}$. Then $s =(\beta,A)$ for some $A\in M_{2}(\mathbb{R})$.
Since $M_{2}(\mathbb{R})$ is regular so there exists $B \in M_{2}(\mathbb{R})$ such that $ABA=A$.
Let us consider $t= (\beta, B)$, then $t \in S$ satisfying $sts=s$.
Now if $s\in T_{1}$ then $s = (\alpha,X)$ where $X \in I$.
Let $X= \begin{pmatrix}
          a & b \\
          0 & 0
       \end{pmatrix}$ where $a, b \in \mathbb{R}$.
Now if $a\neq 0$, take $t= (\alpha,Y)$ where
$Y = \begin{pmatrix}
     \frac{1}{a} & 0 \\
               0 & 0
       \end{pmatrix}$. Then $t \in S$ satisfying $sts=s$. Similarly if $b\neq 0$, take $t= (\beta,Z)$ where
$Z = \begin{pmatrix}
                        0 & 0 \\
               \frac{1}{b} & 0
       \end{pmatrix}$. Then $t \in S$ satisfying $sts=s$. Hence $S$ is multiplicatively regular.

 Let $s_{1}= (\alpha, C)\in T_{1}$ where $C=\begin{pmatrix}
     0 & 1 \\
     0 & 0
  \end{pmatrix}$.
Then there exists no such $t \in T_{1}$ such that $s_{1}ts_{1}=s_{1}$. So, $T_{1}$ is not a multiplicatively regular skew-ring.
\end{example}
\begin{remark}
    In each of the above examples, we exhibit multiplicatively regular GLCR seminearring having one non-regular $\mathcal{H}^{+}$ class.
It may appear that each of those seminearrings may have some other decomposition into union of near-rings such that each component near-ring is regular.
But semigroup theoretic argument on the additive structure of a GLCR seminearring $S$ ensures that whatever be the decomposition of $S$ into union of near-rings, each of the component near-rings must be some $\mathcal{H}^{+}$-class of $S$.
\end{remark}
Now it becomes relevant to obtain necessary and sufficient conditions for the component near-rings of a GLCR seminearring to be regular. This is accomplished in the following two results.

\begin{theorem}\label{theorem mult reg}
Let $S$ be a seminearring. Then the following statements are equivalent:
  \begin{itemize}
    \item [(1)] $S$ is generalized left completely regular (GLCR) in which for each $a\in S$ there exists $b\in S$ such that
      ($A1$) $aba=a$,
      ($A2$) $bab=b$ and
      ($A3$) $(a+a^*)b$, $b(a+a^*)\in \mathcal{H}_{ba}^+$,
    where according to Lemma 2.22 of \cite{Tuhin2}, for each $a\in S$, $a^*$ denotes the unique element in $\mathcal{H}_{a}^+$ such that $a + a^* + a$ $=$ $a$, $a + a^*$ $=$ $a^* + a$ and $(a + a^*)a$ $=$ $a + a^*$
    \item [(2)] Every $\mathcal{H}^+$ -class is a regular near-ring;
    \item [(3)] $S$ is a union (disjoint) of regular near-rings.
  \end{itemize}
\end{theorem}

\begin{proof}
  $(1)\Rightarrow (2)$:
  Suppose $(1)$ holds.
  Then by Theorem $2.19$ of \cite{Tuhin2} every $\mathcal{H}^+$ -class is a near-ring.
  Let $a\in S$. Then there exists $b\in S$ satisfying the conditions $(A1)$, $(A2)$ and $(A3)$.
  Then as $\mathcal{H}^{+}$ is a right congruence ({\it cf.} Lemma 2.20 of \cite{Tuhin2}),
  in view of $(A3)$ we get $(a+a^*)ba$~~$\mathcal{H}^+$~~$b(a+a^*)a$,
  {\it i.e.}, $(a+a^*)ba$~~$\mathcal{H}^+$~~$b(a+a^*)$ (since $(a+a^*)a=a+a^*$). Hence in view of $(A3)$, $(a+a^*)ba$~~$\mathcal{H}^+$~~$(a+a^*)b$,
  {\it i.e.}, $(a+a^*)ba$~~$\mathcal{H}^+$~~$(a+a^*)ab$ (as $(a+a^*)a=a+a^*$).
  Since both $(a+a^*)ba$, $(a+a^*)ab$ are additive idempotents and every $\mathcal{H}^+$ -class is a near-ring, we have $(a+a^*)ba$ = $(a+a^*)ab=(a+a^*)b$ (using $(a+a^*)a=a+a^*$).
  Again as $(a+a^*)a=a+a^*$ and $aba=a$, we deduce that $(a+a^*)b$ = $(a+a^*)ba$ = $((a+a^*)a)ba$ = $(a+a^*)a$ = $a+a^*\in \mathcal{H}_{a}^+$.
  Hence $(a+a^*)b$~~$\mathcal{H}^+$~~$a$.
  Now as $(a+a^*)$~~$\mathcal{H}^+$~~$a$ and $\mathcal{H}^{+}$ is a right congruence, $(a+a^*)b$~~$\mathcal{H}^+$~~$ab$, whence $ab$~~$\mathcal{H}^+$~~$a$.
  Again in view of $(A3)$ and the fact that $(a+a^*)b$~~$\mathcal{H}^+$~~$a$, we see that $ba$~~$\mathcal{H}^+$~~$a$.
  Hence $bab$~~$\mathcal{H}^+$~~$ab$,
  {\it i.e.}, $b$~~$\mathcal{H}^+$~~$ab$ using $(A2)$. So $a$~~$\mathcal{H}^+$~~$b$.
  Thus $\mathcal{H}_{a}^+$ becomes a regular near-ring following Definition 9.153 \cite{Pilz}.
$(2)$ implies $(3)$ trivially.
$(3)\Rightarrow (1)$: Let $S$ be union (disjoint) of regular near-rings $\{N_\alpha:\alpha\in\Lambda\}$.
In view of Theorem 2.19 of \cite{Tuhin2} and regularity of each $(N_\alpha,\cdot)$, it follows that $S$ is a multiplicatively regular, GLCR seminearring.
Let $a\in S$. Then $a\in N_\beta$ for some $\beta\in\Lambda$.
As $N_\beta$ is a regular near-ring, there exists $b\in N_\beta$ such that $aba=a$, $bab=b$.
Also $\mathcal{H}_{a}^+$ = $\mathcal{H}_{b}^+$ = $N_\beta$.
Now in view of Lemma 2.22 of \cite{Tuhin2}, there exists unique $a^*$, $b^*\in N_\beta$ such that $a+a^*$ = $a^*+a$ = $b^*+b$ = $b+b^*$ = $0_{_{N_{\beta}}}$ (where $0_{_{N_{\beta}}}$
denotes the additive identity of $N_{\beta}$).
Thus $ba$, $(a+a^*)b$ and $b(a+a^*)$, being members of $N_\beta$, are all related under the relation $\mathcal{H}^+$.
\end{proof}

The following result is the zero-symmetric counterpart of Theorem \ref{theorem mult reg} whose proof follows from Theorem 2.25 \cite{Tuhin2} and Theorem \ref{theorem mult reg}.

\begin{theorem}\label{theorem mult reg zero symm version}
Let $S$ be a seminearring. Then the following statements are equivalent:
  \begin{itemize}
    \item [(1)] $S$ is both generalized left completely regular (GLCR) and generalized right completely regular (GRCR) in which for each $a\in S$ there exists $b\in S$ such that
      ($A1$) $aba=a$,
      ($A2$) $bab=b$, and
      ($A3$) $(a+a^*)b$, $b(a+a^*)\in \mathcal{H}_{ba}^+$.
    \item [(2)] Every $\mathcal{H}^+$ -class is a zero-symmetric regular near-ring;
    \item [(3)] $S$ is a union (disjoint) of zero-symmetric regular near-rings.
  \end{itemize}
\end{theorem}

Now we put different types of restrictions on regularity and see that the situation becomes nice compared to regularity case in the sense that the respective regularity of multiplicative structure of a GLCR seminearring implies and implied by that of the component near-rings. 
As a first result in this direction the following theorem records that the multiplicative reduct of a GLCR seminearring is an inverse semigroup if and only if each component near-ring is (multiplicatively) inverse. 

\begin{theorem}\label{theorem mult inverse}
Let $S$ be a seminearring. Then the following statements are equivalent:
  \begin{itemize}
    \item [(1)] $S$ is generalized left completely regular (GLCR) as well as multiplicatively inverse seminearring;
    \item [(2)] Every $\mathcal{H}^+$ -class is an inverse near-ring;
    \item [(3)] $S$ is a union (disjoint) of inverse near-rings.
  \end{itemize}
\end{theorem}

\begin{proof}
  $(1)\Rightarrow (2)$:
  Suppose $(1)$ holds. Let $a\in S$.
  Then since in view of Theorem 2.19 of \cite{Tuhin2} every $\mathcal{H}^+$ -class is a near-ring, we have $a$~~$\mathcal{H}^+$~~$a^2$.
  Also there exists a unique $b\in S$ satisfying the conditions $aba=a$ and $bab=b$.
  Now since $\mathcal{H}^+$ is a right congruence, $ab^2a$~~$\mathcal{H}^+$~~$a^2b^2a$, {\it i.e.}, $ab^2a$~~$\mathcal{H}^+$~~$a(ab)(ba)$ which can be written as $ab^2a$~~$\mathcal{H}^+$~~$a(ba)(ab)$ (since $ab$ and $ba$ both are idempotents of the inverse semigroup $(S,\cdot)$).
  Hence $ab^2a$~~$\mathcal{H}^+$~~$a^2b$. Again in view of the facts that $a$~~$\mathcal{H}^+$~~$a^2$ and $\mathcal{H}^+$ is a right congruence we have $a^2b$~~$\mathcal{H}^+$~~$ab$. So $(ab, ab^2a) \in \mathcal{H}^+$.
  Also since $b$~~$\mathcal{H}^+$~~$b^2$, we have $ba^2b$~~$\mathcal{H}^+$~~$b^2a^2b$, {\it i.e.}, $ba^2b$~~$\mathcal{H}^+$~~$b(ba)(ab)$,
  {\it i.e.}, $ba^2b$~~$\mathcal{H}^+$~~$b(ab)(ba)$ whence $ba^2b$~~$\mathcal{H}^+$~~$b^2a$. Also $b^2a$~~$\mathcal{H}^+$~~$ba$.
  Thus $(ba^2b, ba)\in \mathcal{H}^+$.
  Now since $ab$ and $ba$ both are idempotents of the inverse semigroup $(S,\cdot)$, $ab^2a$ $=$ $ba^2b$ whence $ab$ $\mathcal{H}^+$ $ba$.
  Again $(a+a^*)\cdot b$~$\in$~$\mathcal{H}_{ab}^+$ and $(b+b^*)\cdot a$~$\in$~$\mathcal{H}_{ba}^+$ $=$ $\mathcal{H}_{ab}^+$, where according to Lemma 2.22 of \cite{Tuhin2}, for each $a\in S$, $a^*$ denotes the unique element in $\mathcal{H}_{a}^+$ such that $a + a^* + a$ $=$ $a$, $a + a^*$ $=$ $a^* + a$ and $(a + a^*)a$ $=$ $a + a^*$.
  Also $(a+a^*)\cdot b$, $(b+b^*)\cdot a$ $\in$ $E^+(S)$. Since every $\mathcal{H}^+$ class contains a unique additive idempotent, therefore $(a+a^*)\cdot b$ $=$ $(b+b^*)\cdot a$.
  Now $(a+a^*)$, $(b+b^*)$, being zero elements of some $\mathcal{H}^+$ classes, are idempotents in the inverse semigroup $(S,\cdot)$ and so $(a+a^*)$ $(b+b^*)$ $=$ $(b+b^*)$ $(a+a^*)$
  Therefore, $(a+a^*)$ $(b+b^*)a$ $=$ $(b+b^*)$ $(a+a^*)a$,
  {\it i.e.}, $(a+a^*)$ $(a+a^*)b$ $=$ $(b+b^*)$ $(a+a^*)a$ (as $(a+a^*)\cdot b$ $=$ $(b+b^*)\cdot a$),
  {\it i.e.}, $(a+a^*)b$ $=$ $(b+b^*)$ $(a+a^*)$.
  Again as $\mathcal{H}^+$ is a right congruence, we have $(b+b^*)(a+a^*)$ $\mathcal{H}^+$ $b(a+a^*)$.
  Hence $(a+a^*)b$ $\mathcal{H}^+$ $b(a+a^*)$.
  Then in view of the proof of $(1)\Rightarrow (2)$ of Theorem \ref{theorem mult reg}, we have $(a,ab)$, $(ab,ba)\in \mathcal{H}^+$.
Therefore, $b$ $=$ $bab$~~$\mathcal{H}^+$~~$a$.
$(2)\Rightarrow (3)$ is obvious.
$(3)\Rightarrow (1)$ holds in view of Theorem 2.19 of \cite{Tuhin2}.
\end{proof}

The following result is the zero-symmetric counterpart of Theorem \ref{theorem mult inverse} whose proof follows from Theorem 2.25 \cite{Tuhin2} and Theorem \ref{theorem mult inverse}.

\begin{theorem}\label{theorem mult inverse zero symm version}
Let $S$ be a seminearring. Then the following statements are equivalent:
  \begin{itemize}
    \item [(1)] $S$ is both generalized left completely regular (GLCR) and generalized right completely regular (GRCR) which is multiplicatively inverse seminearring as well;
    \item [(2)] Every $\mathcal{H}^+$ -class is a zero-symmetric, inverse near-ring;
    \item [(3)] $S$ is a union (disjoint) of zero-symmetric, inverse near-rings.
  \end{itemize}
\end{theorem}

The following theorem is the completely regular analogue of Theorem \ref{theorem mult inverse}.

\begin{theorem}\label{theorem 1}
   Let $S$ be a seminearring. Then the following statements are equivalent:
  \begin{itemize}
    \item [(1)] $S$ is generalized left completely regular (GLCR) as well as multiplicatively completely regular;
    \item [(2)] Every $\mathcal{H}^+$ -class is a completely regular near-ring;
    \item [(3)] $S$ is a union (disjoint) of completely regular near-rings.
  \end{itemize}
\end{theorem}
\begin{proof}
  $(1)\Rightarrow (2)$:
  Suppose $(1)$ holds.
  Then by Theorem 2.19 of \cite{Tuhin2} every $\mathcal{H}^+$ -class is a near-ring. Let $a\in S$ and $y \in \mathcal{H}_{a}^+$. Then by Lemma 2.22 of \cite{Tuhin2} there exists a unique element $y^* \in \mathcal{H}_{a}^+$ such that $y + y^* + y$ $=$ $y$, $y + y^*$ $=$ $y^* + y$ and $(y + y^*)y$ $=$ $y + y^*$. Since $S$ is multiplicatively completely regular, there exists a unique $c_{y}\in S$ such that $yc_{y}y$ $=$ $y$, $c_{y}yc_{y}$ $=$ $c_{y}$ and $c_{y}y$ $=$ $yc_{y}$.
  Hence $(y + y^* + y)c_{y}y$ = $yc_{y}y$, {\it i.e.},
  $y + y^*c_{y}y + y = y.$
  Also,
$(y^* + y + y^*)c_{y}y$ = $y^*c_{y}y$
 implies that $y^*c_{y}y + y + y^*c_{y}y = y^*c_{y}y$.
   Now $y \mathcal{H^+} y^*$. Since $\mathcal{H}^+$ is a right congruence ({\it cf.} Lemma 2.20 of \cite{Tuhin2}), $yc_{y}y\mathcal{H^+}y^*c_{y}y$. So $y\mathcal{H^+}y^*c_{y}y$. 
  As $y^*$, $y\in \mathcal{H}_{a}^+$ and $(\mathcal{H}_{a}^+,+,\cdot)$ is a near-ring, so $y^*y\in \mathcal{H}_{a}^+$.
  Now $y^*y$ $\mathcal{H}^+$ $y$ implies that $y^*yc_{y}$ $\mathcal{H}^+$ $yc_{y}$ ($\because\mathcal{H^+}$ is a right congruence).
Hence $y^*c_{y}y$ $\mathcal{H}^+$ $yc_{y}$ ($\because yc_{y}=c_{y}y$).
Thus $y$ $\mathcal{H}^+$ $yc_{y}$ ($\because$ $y\mathcal{H^+}y^*c_{y}y$).
Therefore, $y$ $\mathcal{H}^+$ $c_{y}y$
whence $yc_{y}$ $\mathcal{H}^+$ $c_{y}yc_{y}$.
So $y$ $\mathcal{H}^+$ $yc_{y}$ and $yc_y$ $\mathcal{H}^+$ $c_{y}$.
 Therefore $c_{y}\in \mathcal{H}_{a}^+$. Hence $(\mathcal{H}_{a}^+, +, \cdot)$ is a completely regular near-ring.
$(2)\Rightarrow (3)$ is obvious.
$(3)\Rightarrow (1)$ holds in view of Theorem 2.19 of \cite{Tuhin2}.
\end{proof}
The following result is the zero-symmetric counterpart of Theorem \ref{theorem 1} whose proof follows from Theorem 2.25 \cite{Tuhin2} and Theorem \ref{theorem 1}.
\begin{theorem}\label{theorem 2}
   Let $S$ be a seminearring. Then the following statements are equivalent:
  \begin{itemize}
    \item [(1)] $S$ is both generalized left completely regular (GLCR) and generalized right completely regular (GRCR) as well as multiplicatively completely regular;
    \item [(2)] Every $\mathcal{H}^+$ -class is a completely regular, zero-symmetric near-ring;
    \item [(3)] $S$ is a union (disjoint) of completely regular, zero-symmetric near-rings.
  \end{itemize}
\end{theorem}
To conclude the paper we obtain the following result, the Clifford analogue of Theorem \ref{theorem mult inverse}, by combining Theorems \ref{theorem mult inverse}, \ref{theorem mult inverse zero symm version}, \ref{theorem 1}, \ref{theorem 2} together with the fact that every Clifford near-ring is zero-symmetric.
\begin{theorem}\label{theorem Clifford}
Let $S$ be a seminearring. Then the following statements are equivalent:
  \begin{itemize}
   \item [(1)] $S$ is GLCR as well as multiplicatively Clifford seminearring;
    \item [(2)] $S$ is both GLCR , GRCR as well as multiplicatively Clifford seminearring;
    \item [(3)] Every $\mathcal{H}^+$ -class is a Clifford near-ring;
    \item [(4)] $S$ is a union (disjoint) of Clifford near-rings.
  \end{itemize}
\end{theorem}

\end{document}